\documentclass[11pt]{article}
\usepackage{tocbibind} 
\usepackage{amsmath,amsthm,amssymb}
\usepackage{fancyhdr}
\usepackage[british]{babel}
\usepackage{geometry,mathtools}
\usepackage{enumitem}
\usepackage{algpseudocode}
\usepackage{dsfont}
\usepackage{centernot}
\usepackage{xstring}
\usepackage{colortbl}

\usepackage[section]{algorithm}
\usepackage{graphicx}
\usepackage[font={footnotesize}]{caption}
\usepackage[usenames,dvipsnames,table]{xcolor}
\usepackage{pstricks,tikz}
\usetikzlibrary{patterns,calc,fadings,decorations.pathreplacing}
\usepackage[h]{esvect}
\usepackage[
bookmarksopen=true,
bookmarksopenlevel=1,
colorlinks=true,
linkcolor=darkblue,
linktoc=page,
citecolor=darkblue,
urlcolor=darkblue
]{hyperref}
\usepackage{bbm}
\usepackage{xskak}

\numberwithin{equation}{section}

\definecolor{darkblue}{rgb}{0,0,0.5}

\newdimen\margin
\def\textno#1&#2\par{
	\margin=\hsize
	\advance\margin by -4\parindent
	\setbox1=\hbox{\sl#1}
	\ifdim\wd1 < \margin
	$$\box1\eqno#2$$
	\else
	\bigbreak
	\hbox to \hsize{\indent$\vcenter{\advance\hsize by -3\parindent
			\it\noindent#1}\hfil#2$}
	\bigbreak
	\fi}

\newtheorem{theorem}[algorithm]{Theorem}
\newtheorem{prop}[algorithm]{Proposition}
\newtheorem{lemma}[algorithm]{Lemma}
\newtheorem{cor}[algorithm]{Corollary}

\theoremstyle{definition}
\newtheorem{problem}[algorithm]{Problem}

\newtheorem{defin}[algorithm]{Definition}

\def\lateproof#1{\removelastskip\penalty55\medskip\noindent\begin{stepenv}\end{stepenv}{\bf Proof of #1. }} 

\def\noproof{{\unskip\nobreak\hfill\penalty50\hskip2em\hbox{}\nobreak\hfill%
		$\square$\parfillskip=0pt\finalhyphendemerits=0\par}\goodbreak}
\def\endproof{\noproof\bigskip}

\newcounter{stepenv}
\newenvironment{stepenv}[1][]{\refstepcounter{stepenv}}{}

\newcounter{step}[stepenv]

\newcounter{substep}[step]
\renewcommand{\thesubstep}{\thestep.\arabic{substep}}

\newcounter{claim}[stepenv]

\newcommand{\cC}{\mathcal{C}}
\newcommand{\cD}{\mathcal{D}}

\newcommand{\cH}{\mathcal{H}}

\newcommand{\cL}{\mathcal{L}}
\newcommand{\cM}{\mathcal{M}}

\newcommand{\cR}{\mathcal{R}}

\newcommand{\bN}{\mathbb{N}}

\newcommand{\bR}{\mathbb{R}}

\def\eps{{\epsilon}}

\newcommand{\defn}{\emph}

\newcommand{\prob}[1]{\mathrm{\mathbb{P}}\left[#1\right]}

\newcommand{\expn}[1]{\mathrm{\mathbb{E}}\left[#1\right]}

\newcommand{\set}[2]{\{#1\,:\;#2\}}
\def\In{\subset}

\newcommand{\qc}{{\rm qc}}

\def\COMMENT#1{}
\def\TASK#1{}
\let\TASK=\footnote             

\def\x{1}

\begin{document}

	\title{The $n$-queens completion problem}

	\author{Stefan Glock \thanks{Institute for Theoretical Studies, ETH, 8092 Z\"urich, Switzerland.
			Email: \href{mailto:dr.stefan.glock@gmail.com}{\nolinkurl{dr.stefan.glock@gmail.com}}.
			\newline Research supported by Dr. Max R\"ossler, the Walter Haefner Foundation and the ETH Z\"urich Foundation.}
		\and 
		David Munh\'a Correia \thanks{Department of Mathematics, ETH, 8092 Z\"urich, Switzerland. \newline Email:
			\href{mailto:david.munhacanascorreia@math.ethz.ch} {\nolinkurl{david.munhacanascorreia@math.ethz.ch}}.}
		\and
		Benny Sudakov \thanks{Department of Mathematics, ETH, 8092 Z\"urich, Switzerland. Email:
			\href{mailto:benny.sudakov@gmail.com} {\nolinkurl{benny.sudakov@gmail.com}}.
			\newline Research supported in part by SNSF grant 200021\_196965.}
	}
	
	\date{}
	
	\maketitle
	
	\begin{abstract} 
		An $n$-queens configuration is a placement of $n$ mutually non-attacking queens on an $n\times n$ chessboard. The $n$-queens completion problem, introduced by Nauck in 1850, is to decide whether a given partial configuration can be completed to an $n$-queens configuration. In this paper, we study an extremal aspect of this question, namely:
		how small must a partial configuration be so that a completion is always possible? We show that any placement of at most $n/60$ mutually non-attacking queens can be completed. We also provide partial configurations of roughly $n/4$ queens that cannot be completed, and formulate a number of interesting problems. Our proofs 
		connect the queens problem to rainbow matchings in bipartite graphs and use probabilistic arguments together with linear programming duality. 
	\end{abstract}

	\section{Introduction}
	
	An \defn{$n$-queens configuration} is a set of $n$ queens on an $n\times n$ chessboard such that no two are contained in the same row, column, or diagonal.
	The \defn{$n$-queens problem} is to determine $Q(n)$, the number of distinct $n$-queens configurations. The 8-queens problem was first published by German chess composer Max Bezzel in 1848, and attracted the attention of many mathematicians, including Gauss.
	We refer to the surveys~\cite{BS:09,RVZ:94} for a more detailed account of the history and many related problems. 
	Recently, there has been some exciting progress~\cite{BK:21,luria:17,LS:21,simkin:21} towards the $n$-queens problem and its toroidal version which was introduced by P\'olya in 1918.

	In this paper, we study a very natural variant, namely the $n$-queens completion problem, the first instance of which was introduced by Nauck~\cite{nauck:1850} in 1850 (see Figure~\ref{fig:nauck}).
	\begin{figure}[ht]
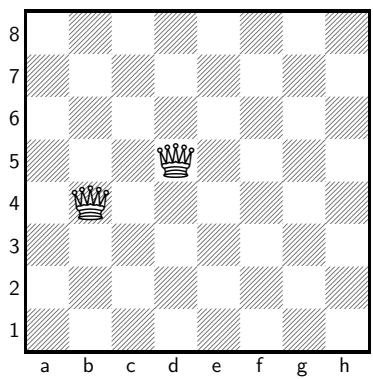

		\fenboard{8/8/8/3Q4/1Q6/8/8/8 w - - 0 0}
		\begin{center}
			\noindent \scalebox{0.8}{\showboard}
		\end{center}
		\caption{In 1850, Nauck posed the problem with two queens on b4 and d5 already given. Is it possible to add six more queens to obtain an $8$-queens configuration?}
		\label{fig:nauck}
	\end{figure}
	In order to formulate the problem precisely, let us represent the chessboard by the $2$-dimensional grid $[n]\times [n]$. For $i\in [n]$, define
	\begin{align*}
		R_i=\set{(i,j)}{j\in [n]},\\
		C_i=\set{(j,i)}{j\in [n]}.
	\end{align*}
	We let $\cR_n=\set{R_i}{i\in [n]}$ denote the set of all rows and $\cC_n=\set{C_i}{i\in [n]}$ the set of all columns.
	For $k\in\{-(n-1),\dots,n-1\}$, define 
	\begin{align*}
		D^+_k &= \set{(i,j)\in [n]\times [n]}{i+j-(n+1)=k},\\
		D^-_k &= \set{(i,j)\in [n]\times [n]}{i-j=k}.
	\end{align*}
	Hence, $D^+_0$ and $D^-_0$ are the two main diagonals of size~$n$.
	We let $\cD_n=\set{D^+_k ,D^-_k}{k\in\{-(n-1),\dots,n-1\}}$ denote the set of all diagonals (see Figure~\ref{fig:rainbow connection} below).
	Finally, we refer to the elements of $\cL_n=\cR_n \cup \cC_n\cup \cD_n$ simply as \defn{lines}.
	If $n$ is clear from the context, we omit the subscripts.
	
	A \defn{partial $n$-queens configuration} is a subset $Q'\In [n]\times [n]$ such that every line contains at most one element of~$Q'$, and an \defn{$n$-queens configuration} is a partial $n$-queens configuration of size~$n$. We say that a partial $n$-queens configuration $Q'$ can be \defn{completed} if there exists an $n$-queens configuration $Q$ with $Q'\In Q$.
	Now, the \defn{$n$-queens completion problem} is the following: Given a positive integer $n$ and a partial $n$-queens configuration $Q'$, can $Q'$ be completed?
	
	When $Q'$ is empty, the question simply becomes whether an $n$-queens configuration exists. This is known to be the case whenever $n\notin \{2,3\}$. However, the standard constructions (see~\cite{BS:09}) are mostly obtained via some algebraic equations and are thus very rigid. In particular, by placing a few queens, one can quickly rule out such a construction.
	In contrast, the recent works~\cite{BK:21,LS:21,simkin:21} on counting $n$-queens configurations rely on probabilistic methods. Roughly speaking, the idea is to place queens iteratively using a random procedure, where in each step one of the still available squares is chosen at random, until $(1-o(1))n$ queens are placed. Then the absorbing method is used to show that with high probability such a partial configuration can be completed. By carefully analysing this 2-phased procedure, one can obtain lower bounds on~$Q(n)$.
	On a high level, this suggests that if a given partial $n$-queens configuration $Q'$ ``looks random'', it should be completable.
	
	We consider an extremal aspect of the $n$-queens completion problem, namely, how small must $Q'$ be so that a completion is always possible?
	This type of question has a long history and has been extensively studied for many related problems. For instance, Evans~\cite{evans:60} conjectured in 1960 that any partial Latin square of order $n$ with fewer than $n$ cells filled can be completed, which would be best possible. Evan's conjecture was eventually proved by Smetaniuk~\cite{smetaniuk:81} and independently by Andersen and Hilton~\cite{AH:83} (see also~\cite{AZ:18} for a nice exposition about this problem).
	In a similar spirit, given a partial Steiner triple system of order $n\equiv 1,3\mod{6}$, a famous conjecture of Nash-Williams~\cite{nash-williams:70} would imply that if at most $n/4$ pairs are covered at every point, then the triple system can be completed.
	Motivated by this, we make the following definition.
	
	\begin{defin}
		Define $\qc(n)$ as the maximum integer with the property that any partial $n$-queens configuration of size at most $\qc(n)$ can be completed. We call $\qc(n)$ the \defn{$n$-queens completion threshold}.
	\end{defin}
	
	We remark that this parameter also has an important algorithmic consequence.
	The $n$-queens problem has been used in many AI papers as a benchmark problem. Gent, Jefferson and Nightingale~\cite{GJN:17} showed that the $n$-queens completion problem is NP-complete in general. However, observe that for all $Q'$ whose size is restricted to at most $\qc(n)$, the decision problem whether $Q'$ can be completed is trivial.
	We refer to~\cite{GJN:17} for a more thorough discussion of complexity and its ramifications for Artificial Intelligence research.

	Our main result is that the completion threshold is linear in~$n$.
	
	\begin{theorem}\label{thm:main}
		For all sufficiently large $n$, we have $n/60\le \qc(n)\le n/4.$
	\end{theorem}
	
	To prove this theorem, in Section~\ref{sec:lower bound}, we establish a connection between the $n$-queens completion problem and rainbow matchings in bipartite graphs.
	Then we deduce our lower bound modulo the proof of a ``rainbow matching lemma''. The study of rainbow subgraph problems has been very fertile in recent years (see e.g.~\cite{AB:09,APS:17,CKPY:20,CP:19,EGJ:19b,GRWW:21,GKMO:21,KPSY:ta,KKKO:20,MPS:19,MPS:20,MPS:20pre,MCS:21,pokrovskiy:18}).
	Our new rainbow matching lemma (Lemma~\ref{lem:rainbow matching}) relates to this large body of recent results in Extremal Combinatorics. Its proof, provided in Section~\ref{sec:matching}, extends ideas of several earlier works. Besides proving the tool we need, our aim in this section is to discuss these important ideas and show how to use them to give a streamlined and mostly self-contained argument. The upper bound in Theorem~\ref{thm:main} is obtained in Section~\ref{sec:construction} via an explicit construction. The proof that our constructed partial configuration is not completable uses linear programming duality.
	We have not attempted to optimize the constants in either bound. Some of our arguments can certainly be improved to yield better bounds. However, new ideas are needed to determine $\qc(n)$ asymptotically. We discuss this and some other possible avenues for future research in the final Section~\ref{sec:remarks}.

	\section{Queens completion problem and rainbow matchings} \label{sec:lower bound}
	
	In this section, we establish our lower bound on $\qc(n)$, by proving that any partial $n$-queens configuration of size at most $n/60$ can be completed.
	
	One of the main ideas is to formulate the problem as a rainbow matching problem. This allows us to make use of a rich pool of tools and methods which have been developed for similar questions.
	A lot of research in this direction has been motivated by the famous Ryser--Brualdi--Stein conjecture. Originally formulated as a question about transversals in Latin squares, it postulates that every optimal edge colouring of $K_{n,n}$ contains a matching of size $n-1$ which is \defn{rainbow}, that is, all edges in it have distinct colours. The conjecture is still wide open. The best known bound, guaranteeing a rainbow matching of size $n-O(\log n /\log\log n)$, was recently obtained by Keevash, Pokrovskiy, Sudakov and Yepremyan~\cite{KPSY:ta} and improved several earlier results.
	
    The $n$-queens problem can be formulated in a very similar way.
	Given the $n\times n$ chessboard, consider the complete bipartite graph with one part $\cR$ consisting of the $n$ rows and one part $\cC$ consisting of the $n$ columns.
	Clearly, the edges of this graph correspond one-to-one to the squares of the chessboard. Moreover, a matching in the graph can be viewed as a placement of queens with no two queens in the same row or column.
	In order to encode the diagonal conflicts, we assign to each edge $(R_i,C_j)\in \cR\times \cC$ the set $\{D^+_{i+j-(n+1)},D^-_{i-j}\}$, that is, the two diagonals which contain the square $(i,j)$.
	Hence, in analogy with the Ryser--Brualdi--Stein conjecture, we might view the diagonals as colours, but instead of having just one colour per edge, each edge receives a set of two colours (see Figure~\ref{fig:rainbow connection}). 
	
	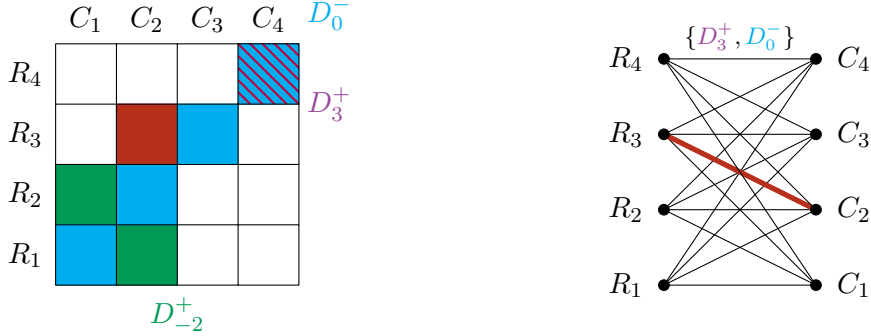
\begin{figure}[ht]
		\begin{center}
			\begin{tikzpicture}
				\begin{scope}[shift={(-4,0)},scale=0.8]

					\draw[fill,BrickRed] (1,2) rectangle (2,3);
					\draw[fill,ForestGreen] (2,0) rectangle (1,1);
					\draw[fill,ForestGreen] (1,1) rectangle (0,2);
					\draw[fill,Cyan] (0,0) rectangle (1,1);
					\draw[fill,Cyan] (1,1) rectangle (2,2);
					\draw[fill,Cyan] (2,2) rectangle (3,3);
					\draw[fill,Cyan] (3,3) rectangle (4,4);
					\begin{scope}
						\clip(3,3) rectangle (4,4);
						\foreach \w in {0,0.2,...,2}
						\draw[purple,thick] (3,3+\w)--(3+\w,3);
					\end{scope}

					\draw[step=1,black,thin] (0,0) grid (4,4);

					\node at (-0.5,0.5) {$R_1$};
					\node at (-0.5,1.5) {$R_2$};
					\node at (-0.5,2.5) {$R_3$};
					\node at (-0.5,3.5) {$R_4$};
					\node at (0.5,4.4) {$C_1$};
					\node at (1.5,4.4) {$C_2$};
					\node at (2.5,4.4) {$C_3$};
					\node at (3.5,4.4) {$C_4$};
					
					\node at (4.5,4.5) {\textcolor{Cyan}{$D^-_0$}};
					\node at (4.5,3) {\textcolor{Purple}{$D^+_3$}};
					\node at (2,-0.5) {\textcolor{ForestGreen}{$D^+_{-2}$}};

				\end{scope}	
				\begin{scope}[shift={(4,0)}, scale=1]
					\coordinate (x1) at (0,0);
					\coordinate (x2) at (0,1);
					\coordinate (x3) at (0,2);
					\coordinate (x4) at (0,3);
					\coordinate (y1) at (2,0);
					\coordinate (y2) at (2,1);
					\coordinate (y3) at (2,2);
					\coordinate (y4) at (2,3);

					\draw (x1)--(y1);
					\draw (x1)--(y2);
					\draw (x1)--(y3);
					\draw (x1)--(y4);
					\draw (x2)--(y1);
					\draw (x2)--(y2);
					\draw (x2)--(y3);
					\draw (x2)--(y4);
					\draw (x3)--(y1);
					\draw[BrickRed,line width=2pt] (x3)--(y2);
					\draw (x3)--(y3);
					\draw (x3)--(y4);
					\draw (x4)--(y1);
					\draw (x4)--(y2);
					\draw (x4)--(y3);
					\draw (x4)--(y4);
					
					\draw[fill] (x1) circle (2pt);
					\draw[fill] (x2) circle (2pt);
					\draw[fill] (x3) circle (2pt);
					\draw[fill] (x4) circle (2pt);
					\draw[fill] (y1) circle (2pt);
					\draw[fill] (y2) circle (2pt);
					\draw[fill] (y3) circle (2pt);
					\draw[fill] (y4) circle (2pt);
					
					\node at (-0.5,0) {$R_1$};
					\node at (-0.5,1) {$R_2$};
					\node at (-0.5,2) {$R_3$};
					\node at (-0.5,3) {$R_4$};
					\node at (2.5,0) {$C_1$};
					\node at (2.5,1) {$C_2$};
					\node at (2.5,2) {$C_3$};
					\node at (2.5,3) {$C_4$};
					
					\node at (1,3.3) {\footnotesize{$\{\textcolor{Purple}{D^+_3},\textcolor{Cyan}{D^-_0}\}$}};
				\end{scope}
			\end{tikzpicture}
		\end{center}
		\caption{Representing the chessboard as a bipartite graph. In the complete bipartite graph $K_{n,n}$, every edge represents a square of the chessboard. Every edge is ``coloured'' with a set of size 2 consisting of the two diagonals which contain the corresponding square.}\label{fig:rainbow connection}
	\end{figure}

	Observe that an $n$-queens configuration corresponds exactly to a perfect matching which is rainbow, meaning now that the colour sets of the edges in the matching are pairwise disjoint.
	Moreover, given a partial $n$-queens configuration $Q'$, let $\cR'$ be the set of remaining rows and $\cC'$ the set of remaining columns, and let $G$ be the bipartite graph with parts $\cR',\cC'$ where $(R_i,C_j)\in \cR'\times \cC'$ is still an edge if $(i,j)$ is not diagonally  attacked by any queen from~$Q'$. 
	Then a completion of $Q'$ corresponds to a perfect rainbow matching of~$G$.
	Note that assuming $Q'$ is not too large, $G$ will still be relatively dense. Indeed, say we consider $R_i\in \cR'$, then the degree of $R_i$ in $G$ will be at least $n-3|Q'|$ since every queen in $Q'$ attacks at most $3$ squares from~$R_i$ (one in the same column and up to two diagonally).
	
	Motivated by this, we want to understand under what conditions a dense graph $G$ with a certain edge colouring has a perfect rainbow matching.
	For simplicity, let us first consider again the case when $G$ is the complete bipartite graph $K_{n,n}$.
	Notice that the Ryser--Brualdi--Stein conjecture only asks for a rainbow matching of size $n-1$. In fact, when $n$ is even, there are examples of optimal edge colourings of $K_{n,n}$ without a perfect rainbow matching.
	A fair amount of research has been devoted to find additional conditions which guarantee a perfect rainbow matching.
	One natural additional assumption is to restrict the degrees of the colours, where by \defn{degree} of a colour we mean the number of edges with that colour.
	For instance, Erd\H{o}s and Spencer~\cite{ES:91} showed, by developing the lopsided Lov\'asz local lemma, that if $K_{n,n}$ is properly edge-coloured and all colour degrees are at most $n/16$, then a perfect rainbow matching exists.
	Recently, Montgomery, Pokrovskiy and Sudakov~\cite{MPS:19} and independently Kim, K\"uhn, Kupavskii and Osthus~\cite{KKKO:20} significantly strengthened this by proving that it suffices that all colour degrees are at most $(1-o(1))n$. 
	
	\begin{theorem}[\cite{KKKO:20,MPS:19}]\label{thm:matching}
		For every $\alpha>0$, if $n$ is sufficiently large, every properly edge coloured $K_{n,n}$ with all colour degrees at most $(1-\alpha)n$ has a perfect rainbow matching.
	\end{theorem}
	
	We will prove here a ``rainbow matching lemma'' which is of the same nature as Theorem~\ref{thm:matching} and suited to our needs. 
	In order to state it, we need to introduce a bit more notation.
	Given a graph $G$, a \defn{$t$-fold edge colouring} is an assignment of sets of size $t$ to the edges of~$G$.
	The colouring is called \defn{proper} if the edges at a given vertex have pairwise disjoint colour sets, and a subgraph of $G$ is called \defn{rainbow} if all its edges have pairwise disjoint colour sets. 
	In our application, we will only need $2$-fold edge colourings (a square of the chessboard coloured with the two diagonals containing it). 
	The \defn{degree} of a colour $c$ is the number of edges whose colour set contains~$c$.
	Finally, we call the colouring \defn{linear} if for every pair of colours there is at most one edge that contains both.
	
	As discussed above, it is important for us to find perfect rainbow matchings in graphs which are not necessarily complete bipartite.
	The crucial conditions (which are trivially satisfied by $K_{n,n}$) are that the given graph is ``approximately regular'', meaning that all vertex degrees are roughly the same, and ``uniformly dense'' in the sense that between any sufficiently large sets there are many edges.

	\begin{lemma}[rainbow matching lemma]\label{lem:rainbow matching}
		For any $\alpha>0$ and $t\in \bN$ there exist $\eps>0$ and $n_0$ such that the following is true for any $n\ge n_0$.
		Let $G$ be a bipartite graph with parts $A,B$ of size~$n$ with a proper linear $t$-fold edge colouring. Assume that the following conditions are satisfied for some $d$:
		\begin{enumerate}[label=\rm{(\roman*)}]
			\item every vertex has degree $(1\pm \eps)d$;\label{cond vertex degrees}
			\item every colour has degree at most $(1-\alpha)d$;\label{cond colour degrees}
			\item any two sets $A'\In A$ and $B'\In B$ of size at least $(1-\alpha)d$ have at least $\alpha n^2$ edges between them.\label{cond dense}
		\end{enumerate}
		Then $G$ has a perfect rainbow matching.
	\end{lemma}

	Let us briefly compare this result to earlier works. Conceptionally, the main difference is that we consider $t$-fold edge colourings instead of just normal edge colourings. A more subtle variation concerns condition~\ref{cond dense}. Previous proofs required here that even sets of relatively small size, say, $\eps n$, have many edges between them. In contrast, note that the size condition in \ref{cond dense} is asymptotically optimal. For instance, when $G$ is the disjoint union of two copies of $K_{d,d+1}$, where $n=2d+1$, then any sets $A',B'$ of size $(1+\alpha)d$ have quadratically many edges between them, but there is no perfect matching. 
	Finally, we remark that the assumption of the colouring being proper and linear can be weakened. Say the colouring is \defn{$b$-bounded} if every colour appears on at most $b$ edges incident to any given vertex, and any pair of colours appears on at most $b$ edges together. Then our proof of Lemma~\ref{lem:rainbow matching} generalizes to the setting when the given colouring is $b$-bounded for some fixed constant~$b$.

	We will prove Lemma~\ref{lem:rainbow matching} in Section~\ref{sec:matching}. In the remainder of this section, we use it to prove the lower bound of Theorem~\ref{thm:main}.
	In order to motivate the first step of our argument, observe that Lemma~\ref{lem:rainbow matching} is not directly applicable to our problem even if $Q'$ is empty and hence $G$ is the complete bipartite graph $K_{n,n}$. Although conditions~\ref{cond vertex degrees} and \ref{cond dense} trivially hold with $d=n$, condition~\ref{cond colour degrees} is not satisfied because the main diagonals have size $n$ as well.
	However, on average the colours (diagonals) have degree $\sim n/2$. Our first goal is thus to find a subgraph $G'$ where not only the average degree of the colours is significantly smaller than the vertex degrees, but really all the colours have small degree.
	Equivalently, we want to find a subset $\Lambda\In [n]\times [n]$ of the chessboard such that the diagonals contain significantly fewer elements from $\Lambda$ than the rows and columns. 
	It is beneficial to consider the following fractional variant of this problem, where the aim is to find a non-negative weight function on $[n]\times [n]$ such that the total weight of each diagonal is smaller than that of the rows and columns.
	One can easily find such a weight function for $n=3$, say, and then scale it appropriately (see Figure~\ref{fig:weigthing}).

	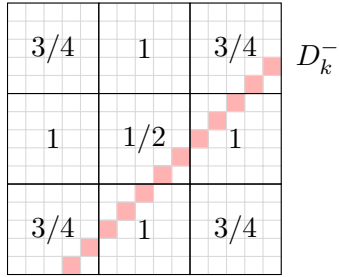
\begin{figure}[ht]
		\begin{center}
			\begin{tikzpicture}[scale=1.2]
				
				\foreach \i in {0,0.2,...,2.2}
				\draw[fill,red!30] (-0.4+\i,-1+\i) rectangle (-0.4+0.2+\i,-1+0.2+\i);
				
				\draw[thin,step=0.2,gray!30] (-1,-1) grid (2,2);

				\draw (0,0) rectangle node{$1/2$} (\x,\x);
				\draw (0,\x) rectangle node{$1$} (\x,\x+\x);
				\draw (\x,0) rectangle node{$1$} (\x+\x,\x);
				\draw (-\x,0) rectangle node{$1$} (0,\x);
				\draw (0,-\x) rectangle node{$1$} (\x,0);
				\draw (\x,\x) rectangle node{$3/4$} (\x+\x,\x+\x);
				\draw (-\x,-\x) rectangle node{$3/4$} (0,0);
				\draw (-\x,\x+\x) rectangle node{$3/4$} (0,\x);
				\draw (\x,0) rectangle node{$3/4$} (\x+\x,-\x);
				
				\node at (2.4,1.4) {$D^-_k$};
			\end{tikzpicture}
		\end{center}
		\caption{The weight function for $n=3$. All rows and columns have weight $5/2=5n/6$, while all diagonals have weight at most $2=2n/3$. For larger $n$, one can simply scale this weighting and achieve the same effect.}
		\label{fig:weigthing}
	\end{figure}

	\begin{prop}\label{prop:regularize}
		For all $n\in \bN$, there exists a weighting $w\colon [n]\times[n]\to [1/2,1]$ with the property that every row and column has total weight $5n/6+O(1)$, but every diagonal has weight at most $2n/3+O(1)$.
	\end{prop}
	
	\begin{proof}
		Define a weight function $w\colon [n]\times[n]\to [1/2,1]$ by setting
		\begin{align*}
			w((i,j)) := \begin{cases}     
				1/2 & \mbox{if } i/(n+1),j/(n+1)\in [1/3,2/3],\\
				3/4 &  \mbox{if } i/(n+1),j/(n+1)\in [0,1/3)\cup (2/3,1], \\
				1 & \mbox{otherwise.}
			\end{cases}
		\end{align*}
		Note that we scale $i,j$ by $n+1$ for symmetry reasons.
		With this definition, we have $w(R)=5n/6+O(1)$ for all $R\in \cR$ and $w(C)=5n/6+O(1)$ for all $C\in \cC$. Moreover, we claim that $w(D)\le 2n/3+O(1)$ for any $D\in \cD$. (For a line $L$, we use $w(L)$ to denote the sum of the weights of all elements of~$L$.)
		By symmetry of $w$, it suffices to consider $D^-_k$ for fixed $k\in \{0,\dots,n-1\}$.
		By the definition of $w$, the weight of any such diagonal with $k\ge n/3$ is dominated by the one with $k=\lceil n/3\rceil$, since for larger $k$, the size of the diagonals and the weights of the elements are non-increasing as $k$ increases. Finally, for $k\le \lceil n/3\rceil$, we have
		\begin{align*}
			w(D^-_k) &= (n/3-k)\cdot 3/4 + k\cdot 1 + (n/3-k)\cdot 1/2 + k\cdot 1  +  (n/3-k)\cdot 3/4 +O(1)\\
			&=2n/3+O(1),
		\end{align*}
		proving the claim.
	\end{proof}

	Having found such a weighting, the existence of the desired set $\Lambda$ follows readily from a probabilistic argument. Simply choose $\Lambda$ by including every square  $(i,j)$  independently with probability~$w((i,j))$. Then the \emph{expected} degree of a line is simply its total weight under~$w$, and with high probability all the actual degrees will be close to their expectation. (Throughout, we say that an event holds \defn{with high probability} if its probability tends to $1$ when $n\to \infty$.)
	Note that after this initial step, we could apply Lemma~\ref{lem:rainbow matching} to find an $n$-queens configuration.
	However, if a partial configuration $Q'$ is given, then the graph $G$ in which we seek a perfect rainbow matching might not be regular anymore, since a queen in $Q'$ might attack $3$ elements from some row but only $1$ element from another. 
	Hence, in a second step, we want to regularize the degrees of~$G$. We approach this problem again by considering the fractional variant first, and then using the same probabilistic argument as above to turn an appropriate weight function into a subgraph which approximates the weights.
	In fact, it is more convenient to only apply the probabilistic argument once, after both of the above steps have been carried out in the fractional setting.

	\begin{prop}\label{prop:weight shift}
		Let $c,d'> 0$ and let $G$ be a bipartite graph with parts $A,B$ of size~$n$ where any two vertices in the same part have at least $c$ common neighbours. Let $w_0\colon E(G)\to [0,1]$ be an edge weighting such that $\sum_{e\ni u} w_0(e)=\bar{d}\pm d'$ for all vertices $u$, where $\bar{d}=\frac{1}{n}\sum_{e\in E(G)}w_0(e)$. Then there exists an edge weighting $w \colon E(G)\to \bR$ such that the total weight of edges at every vertex is~$\bar{d}$, and $|w(e)-w_0(e)|\le 2d'/c$ for all $e\in E(G)$.
	\end{prop}

	\begin{proof}
		We obtain the desired weighting iteratively. Starting with $w_0$, in each step, given the current weighting $w_i$, we perform a shift of the following form. Suppose $u,v$ are two vertices in the same part, such that $w_i(u)>\bar{d}$ and $w_i(v)<\bar{d}$ (where $w_i(u)$ is the sum of $w_i(e)$ over all edges $e$ incident to~$u$). Set $\eta:=\min\{w_i(u)-\bar{d},\bar{d}-w_i(v)\}$ and let $X$ be the set of common neighbours of $u$ and~$v$. Now, we make the following adjustment to obtain a new weighting~$w_{i+1}$. For all $x\in X$, we set $w_{i+1}(ux)=w_i(ux)-\frac{\eta}{|X|}$ and $w_{i+1}(vx)=w_i(vx)+\frac{\eta}{|X|}$, and $w_{i+1}(e)=w_i(e)$ for all other edges. Hence, the total weight of $u$ decreases by $\eta$, whereas the total weight of $v$ increases by $\eta$, and the total weight of all other vertices remains unchanged. In particular, when performing such a shift, the total weight of all vertices in $A$ remains the same, namely $\bar{d}n$, and similarly for~$B$. Hence, as long as there is some vertex whose weight is not equal to $\bar{d}$, we can find another vertex in the same part and perform a shift like above. Observe also that the number of vertices whose total weight is not equal to $\bar{d}$ is reduced in each step, hence the procedure will eventually stop with a weighting $w$ where all vertices have total weight exactly~$\bar{d}$.

		Consider an arbitrary edge $e=ab$. We need to check that during the procedure we did not change the weight of $e$ by more than $2d'/c$. For this, observe that the total weight of both $a$ and $b$ is changed by at most $d'$, and in each adjustment, we spread the adjustment over at least $c$ common neighbours. Hence, the total change received by any given edge is at most $2d'/c$.
	\end{proof}

	In the probabilistic argument, we will use the following standard concentration inequality. 
	
	\begin{lemma}[Chernoff--Hoeffding bound]\label{lem:chernoff}
		Let $X$ be the sum of $n$ independent Bernoulli random variables. Then for any $\lambda\ge 0$, we have $$\prob{|X-\expn{X}|\ge \lambda} \le 2\exp(-2\lambda^2/n).$$
	\end{lemma}

	\lateproof{Theorem~\ref{thm:main}, lower bound}
	We consider the complete bipartite graph with one part $\cR$ consisting of the $n$ rows and one part $\cC$ consisting of the $n$ columns. Moreover, we define a $2$-fold edge colouring by assigning to each edge $(R_i,C_j)\in \cR\times \cC$ the set $\{D^+_{i+j-(n+1)},D^-_{i-j}\}$, that is, the two diagonals which contain the square $(i,j)$. Crucially, observe that this colouring is proper and linear, because any two lines of the chessboard intersect in at most one square.

	Now, let $Q'\In [n]\times [n]$ be an arbitrary partial $n$-queens configuration of size $\beta n$, where $\beta\le 1/60$.
	Let $\cR'\In \cR$ be the set of rows not containing a queen from $Q'$, and let $\cC'\In \cC$ be the set of columns not containing a queen from~$Q'$.
	Let $G$ be the subgraph induced by $\cR'$ and $\cC'$, after deleting all edges $(R_i,C_j)$ for which $Q'$ contains a queen on either diagonal $D^+_{i+j-(n+1)},D^-_{i-j}$.
	
	We want to prove that $Q'$ can be completed, which is equivalent to showing that $G$ has a perfect rainbow matching.
	Our aim is to apply the rainbow matching lemma. However, $G$ itself is not suitable for such an application, because it might be quite irregular and the degrees of some colours could be significantly larger than those of the vertices.
	Hence, we find a suitable spanning subgraph $G'\In G$ which meets the required conditions. Note that a perfect rainbow matching of $G'$ is also a perfect rainbow matching of~$G$. In order to prepare for an application of Lemma~\ref{lem:rainbow matching}, let $\alpha,\eps>0$ be sufficiently small and $n$ sufficiently large so that the lemma is applicable with $t=2$.
	
	By Proposition~\ref{prop:regularize}, there exists a weighting $w_0\colon \cR\times \cC \to [1/2,1]$ such that 
	$w_0(R_i)=5n/6+O(1)$ and $w_0(C_i)=5n/6+O(1)$ for all $i\in [n]$, and 
	$w_0(D^\pm_k)\le 2n/3+O(1)$ for all $k\in\{-(n-1),\dots,n-1\}$.
	Now, we restrict this weighting to the edges of the graph~$G$. Slightly abusing notation, we denote this restricted weighting again with~$w_0$. Hence, for all $R_i\in \cR'$ and $C_j\in \cC'$, we still have $w_0(R_i),w_0(C_j) \ge 5n/6+O(1)-3\beta n$ since every queen in $Q'$ attacks at most $3$ squares from $R_i,C_j$.
	For the same reason, any two vertices of $G$ in the same part have at least $n-6\beta n$ common neighbours.
	We still have $w_0(D^\pm_k)\le 2n/3+O(1)$ for all $k\in\{-(n-1),\dots,n-1\}$. Hence, for $\beta \leq 1/60$, the weight of the diagonals (colours) is significantly lower than that of the rows and columns (vertices).
	
	Next, we want to use this gap to regularize the weights of the vertices, while still maintaining a slight gap between the degrees of vertices and colours.
	Define 
	\begin{align}
		\bar{d}:=\frac{1}{(1-\beta)n} \sum_{(R_i,C_j)\in E(G)} w_0((R_i,C_j))
	\end{align}
	as the average weight of a vertex.
	We have $5n/6+O(1)-3\beta n \le \bar{d}\le 5n/6+O(1)$, so by applying Proposition~\ref{prop:weight shift} with $d'=3\beta n +O(1)$ and $c=(1-6\beta)n$, we can find a weighting $w\colon E(G)\to \bR$ such that every vertex has total weight $\bar{d}$, and the weight of $w_0$ is changed by at most $$\mu:=2d'/c =6\beta/(1-6\beta) +o(1)\le 1/9+o(1).$$ 
	
	Now, we randomly sparsify the graph $G$ to obtain a subgraph $G'$ which is approximately regular. For each edge $e\in E(G)$, we define
	$$p_e:=\frac{w(e)}{1+\mu}.$$
	This way, we have $p_e\le 1$ and $p_e\ge 1/10$, say.
	Let us include every edge $e\in E(G)$ independently in $G'$ with probability~$p_e$.
	We want to show that with high probability $G'$ has the properties \ref{cond vertex degrees}--\ref{cond dense} which we need to apply the rainbow matching lemma.
	
	Let us consider any vertex~$u$.
	The expected degree of $u$ in $G'$ is $d:=\bar{d}/(1+\mu)$. The Chernoff bound (Lemma~\ref{lem:chernoff}) implies that
	$$\prob{|d_{G'}(u)-d|\ge \eps d} \le 2\exp(-2(\eps d)^2/n)=\exp(-\Omega(n)).$$
	Hence, by a union bound over the at most $2n$ vertices, we see that \ref{cond vertex degrees} holds with high probability. 
	
	Now, consider any colour~$c$. Since the weight of each edge has increased by at most $\mu$ and $c$ appears on at most $n$ edges, the expected degree of $c$ in $G'$ is at most $(2n/3+O(1)+\mu n)/(1+\mu)$.
	Thus, again using Chernoff's bound, we see that with high probability,
	all colour degrees are at most $$\frac{2/3+\mu+\eps}{1+\mu}n \le \frac{(1-\alpha)\bar{d}}{1+\mu} = (1-\alpha)d,$$
	as needed for \ref{cond colour degrees}. To see that the above inequality is correct, note that $2/3+\mu < 5/6-3\beta$ by the upper bound on~$\beta$, with room to spare, so for small enough $\alpha,\eps$ we have $2/3+\mu+\eps \le (1-\alpha)(5/6-3\beta-o(1))$.
	
	Finally, consider two sets $\cR''\In \cR'$ and $\cC''\In \cC'$ of size at least~$(1-\alpha)d \ge n/2$.
	First observe that in $G$, there are at least 
	$|\cR''||\cC''|-2n\cdot \beta n \geq (n/2)^2 - 2n\cdot \beta n \ge n^2 / 5$ edges between $\cR''$ and~$\cC''$. This is because every queen in $Q'$ attacks at most $2n$ squares of the grid given by $\cR''$ and~$\cC''$.
	Since we have $p_e\ge 1/10$ for all edges $e\in E(G)$, we see that 
	$\expn{e_{G'}(\cR'',\cC'')} \ge n^2/50$. By Chernoff's bound,
	$$\prob{e_{G'}(\cR'',\cC'') \le n^2/100} \le 2\exp(-2(n^2/100)^2/n^2)\le \exp(-n^{3/2}).$$
	A union bound over the at most $4^n$ choices for $\cR'',\cC''$ thus shows that with high probability any two such sets have at least $n^2/100$ edges between them, satisfying \ref{cond dense}.

	Hence, by the probabilistic method, there exists a subgraph $G'$ which has all of the above properties.
	Finally, we can apply Lemma~\ref{lem:rainbow matching} to conclude that $G'$ has a perfect rainbow matching, which corresponds to a completion of~$Q'$.
	\endproof

	\section{Proof of the rainbow matching lemma}\label{sec:matching}
	
	In this section, we prove Lemma~\ref{lem:rainbow matching}.
	As discussed in Section~\ref{sec:lower bound}, this lemma is inspired by Theorem~\ref{thm:matching}. 
	There are two main steps in its proof. In the first step, we find an approximate rainbow matching, that is, a rainbow matching which covers all but $o(n)$ vertices. In the second step, we use local switchings to turn this approximate rainbow matching into a perfect rainbow matching. 
	As a by-product, we also obtain a streamlined proof of Theorem~\ref{thm:matching}.
	
	The first step is achieved using the celebrated R\"odl nibble.
	In order to formalize this, we need to introduce hypergraphs. A \defn{$k$-uniform hypergraph} $\cH$ consists of a set of vertices $V(\cH)$ and a set of edges $E(\cH)$, where every edge is a subset of $V(\cH)$ of size~$k$. 
	It is called \defn{linear} if any two vertices are contained in at most one edge together. 
	A \defn{matching} $\cM$ is a set of disjoint edges. A vertex is \defn{covered} by $\cM$ if it belongs to some edge in~$\cM$. The \defn{degree} of a vertex is simply the number of edges containing it, and the hypergraph is \defn{$D$-regular} if all vertices have degree~$D$.
	
	To see the connection with the $n$-queens problem let us define the hypergraph $\cH$ whose vertices are the lines $\cL_n$ of the chessboard and every square $(i,j)\in [n]\times [n]$ is represented by the edge $\{R_i,C_j,D^+_{i+j-(n+1)},D^-_{i-j}\}$. Then the matchings of $\cH$ correspond exactly to partial $n$-queens configurations.
	
	Many other combinatorial questions can be formulated as matching problems in hypergraphs, which is why general results concerning matchings in hypergraphs are of great interest. For the following discussion, let $\cH$ be a $k$-uniform linear hypergraph, where $k$ is a fixed constant. A celebrated result of Pippenger (see~\cite{PS:89}) based on the R\"odl nibble says that if $\cH$ is $D$-regular, for sufficiently large $D$, then it has an almost perfect matching. (The same conclusion holds if $\cH$ is not necessarily linear, but all codegrees are $o(D)$. For simplicity, we will only consider linear hypergraphs here.)
	
	Notice that the ``chessboard hypergraph'' above is not regular. While all rows and columns (and the two main diagonals) have degree $n$, the other diagonals have smaller degree. 
	What we will need here is a variant of Pippenger's theorem which does not require the hypergraph to be regular, as long as all the vertices we insist to cover have roughly the same degree, and the degrees of all other vertices are smaller (see Theorem~\ref{thm:nibble}). 
	We remark that this observation is by no means novel, but it seems to be less known than the regular case, which is why we include its short derivation from the following classical result.
	
	\begin{theorem}[Pippenger--Spencer~\cite{PS:89}]\label{thm:chromatic index}
		For any $\eps>0$ and $k\in \bN$ there exists $\Delta_0$ such that the following holds for all $\Delta\ge \Delta_0$. Let $\cH$ be a linear $k$-uniform hypergraph with maximum degree~$\Delta$. Then $E(\cH)$ can be partitioned into at most $(1+\eps)\Delta$ matchings.
	\end{theorem}
	
	Pippenger and Spencer originally formulated this result in the regular setting. However, since every linear $k$-uniform hypergraph can be embedded into a regular linear $k$-uniform hypergraph with the same maximum degree, the above version readily follows. Much stronger results have been proved subsequently~\cite{kahn:96,MR:00}.
	
	From this, we can quite easily deduce the variant of Pippenger's theorem which we need. 
	
	\begin{theorem}\label{thm:nibble}
		For any $\delta>0$ and $k\in \bN$ there exists $D_0$ such that the following holds for all $D\ge D_0$.
		Let $\cH$ be a linear $k$-uniform hypergraph with maximum degree at most~$D$. Let $U\In V(\cH)$ be a set of vertices all having degree at least $(1- \delta)D$. 
		Then there exists a matching in $\cH$ which covers all but $2\delta|U|$ of the vertices of~$U$.
	\end{theorem}

	\begin{proof}
		By Theorem~\ref{thm:chromatic index}, there are matchings $M_1,\dots,M_t$ which partition $E(\cH)$ such that $t\le (1+\delta)D$. For every $i\in [t]$ 
		define $X_i=\sum_{e\in M_i}|e\cap U|$ which is the number of vertices in $U$ that are covered by~$M_i$. Then we have 
		$$ \frac{1}{t} \sum_{i=1}^{t} X_i=\frac{1}{t} \sum_{e\in E(\cH)} |e\cap U| =\frac{1}{t}\sum_{u\in U}d_{\cH}(u) \ge \frac{(1-\delta)D|U|}{(1+\delta)D} \ge (1-2\delta)|U|.$$
		Hence, there is $i\in [t]$ for which $X_i\ge (1-2\delta)|U|$, which gives us the desired matching.
	\end{proof}
	
	Using Theorem~\ref{thm:nibble}, we can deduce the following result which gives us an almost-perfect rainbow matching.
	
	\begin{cor}\label{cor:approx}
		For any $t\in \bN$ and $\eps>0$ there exists $d_0$ such that the following holds for all $d\ge d_0$. Let $G$ be a graph with a proper linear $t$-fold edge colouring. Assume that all vertices have degree $(1\pm \eps)d$ and all colours have degree at most $(1+\eps)d$. Then there exists a rainbow matching which covers all but $4\eps |V(G)|$ vertices.
	\end{cor}

	\begin{proof}
		Define the following hypergraph~$\cH$. The vertex set of $\cH$ is $V(G) \cup C$, where $C$ is the set of all colours.
		For every edge $e=uv\in E(G)$, we define the hyperedge $f_e=\{u,v\}\cup C_e$, where $C_e$ is the colour set of~$e$. Hence, $\cH$ is $(t+2)$-uniform.
		
		First, observe that $\cH$ is linear. Indeed, for any distinct $u,v\in V(G)$, if $uv\notin E(G)$, then there is no hyperedge containing $u$ and $v$, and if $uv\in E(G)$, then $f_{uv}$ is the only hyperedge containing $u$ and~$v$. For $u\in V(G)$ and $c\in C$, there is at most one hyperedge containing $u$ and $c$ because the colouring is proper, and for distinct $c,c'\in C$, there is at most one hyperedge containing $c$ and $c'$ since the colouring is linear.
		
		Next, the maximum degree of $\cH$ is at most $(1+\eps)d$ by assumption, and all vertices in $V(G)$ have degree at least $(1-\eps)d$. Apply Theorem~\ref{thm:nibble} (with $D=(1+\eps)d$ and $\delta=2\eps$) to find a matching $\cM$ in $\cH$ which covers all but $4\eps|V(G)|$ vertices of~$V(G)$. Define $M=\set{e\in E(G)}{f_e\in \cM}$. Clearly, $M$ is a matching in $G$ and covers all but $4\eps|V(G)|$ vertices of~$V(G)$. Moreover, for any two edges $e,e'\in M$, their colour sets are disjoint since $f_e$ and $f_{e'}$ are both in~$\cM$. Hence, $M$ is rainbow.
	\end{proof}

	It remains to prove that this almost-perfect matching can be turned into a perfect rainbow matching.
	The crucial assumption to achieve this is that the degrees of the colours are significantly smaller than those of the vertices. This means that we can ``reserve'' some colours which will not be used for the almost-perfect matching, and then in the second step we can use these reserved colours to increase the size of the matching. 
	
	We next prove the tool which allows us to ``split'' the colours, that is, to set aside a set of colours for the final step. This is based on an important observation of Alon, Pokrovskiy and Sudakov~\cite{APS:17}. They showed that, given a properly edge coloured $K_n$, if one chooses every colour class independently with probability $p$, then the obtained subgraph has the property that between any two large enough sets $A,B$ there are at least $(1-o(1))p|A||B|$ edges.
	Building on their ideas, but using a slightly different analysis, we prove the following general tool, which also provides upper bounds on the number of chosen edges.
	
	\begin{lemma}\label{lem:colour split}
		Let $G$ be an $n$-vertex graph with a proper $t$-fold edge colouring.
		Choose every colour independently with probability $p$, and let $G'$ be the (random) subgraph consisting of all the edges~$e$ such that all the colours in the colour set of $e$ have been chosen.
		Then with high probability, for any $m$ and any disjoint sets $A,B\In V(G)$ of size at least $m$, we have that
		$$|e_{G'}(A,B) - p^t e_G(A,B)|\le 9tm^{-1/5}|A||B|\sqrt{\log n}.$$
	\end{lemma}

	\begin{proof}
		For two disjoint sets $A',B'\In V(G)$, we define $Z_{A',B'}$ as the number of pairs of distinct edges $e,e'\in E(G[A',B'])$ whose colour sets are not disjoint.
		
		We claim that with high probability, for any disjoint sets $A',B'\In V(G)$, we have that 
		\begin{align}
			|e_{G'}(A',B') - p^t e_{G}(A',B')| \le Z_{A',B'} + 2 \sqrt{(|A'|+|B'|)|A'||B'|\log n}.\label{individual chernoff}
		\end{align}
		Indeed, fix $A',B'$ and set $a=|A'|$ and $b=|B'|$. By deleting an edge from every pair of edges that have overlapping colour sets, we obtain a rainbow subgraph of $G[A',B']$ with at least $e_{G}(A',B') - Z_{A',B'}$ edges. Let $X$ be the random variable counting the number of these edges which are in~$G'$. Since $X$ is a sum of at most $ab$ independent Bernoulli variables, the Chernoff bound (Lemma~\ref{lem:chernoff}) implies that
		$$\prob{|X-\expn{X}|\ge \lambda} \le 2\exp(-2\lambda^2/ab),$$ for any $\lambda\ge 0$. Set  $\lambda=2 \sqrt{(a+b)ab\log n}$. Since $X\le e_{G'}(A',B')\le X+ Z_{A',B'}$ and $p^t(e_{G}(A',B')-Z_{A',B'})\le \expn{X}\le p^te_{G}(A',B')$, we have $$|e_{G'}(A',B') - p^t e_{G}(A',B')|\le |X-\expn{X}| + Z_{A',B'}.$$ Thus, we see that the probability that \eqref{individual chernoff} is not satisfied is at most $2n^{-8(a+b)}$.
		Since the number of choices for $A'$ and $B'$ with $|A'|=a$ and $|B'|=b$ is at most $n^{a+b}$, a union bound (also over the $n^2$ choices for $a$ and $b$) implies the claim.

		It now suffices to show that~\eqref{individual chernoff} implies the conclusion of the lemma.
		To this end, consider arbitrary disjoint sets $A,B\In V(G)$ of size at least~$m\ge 1$. Set $r=\lceil {|A|}^{3/5}\rceil$ and $s=\lceil {|B|}^{3/5}\rceil$.
		We claim that there exist a partition of $A$ into sets $A_1,\dots,A_r$ and a partition of $B$ into sets $B_1,\dots,B_s$ such that
		\begin{align}
			\sum_{i=1}^r \sum_{j=1}^s Z_{A_i,B_j} \le tm^{-1/5} e_G(A,B)\le tm^{-1/5}|A||B|.\label{partition}
		\end{align}
		Indeed, for each vertex in $A$ and $B$, respectively, put it in one of the sets $A_i$ and $B_i$ uniformly at random, all independently.
		Then $$\expn{\sum_{i=1}^r \sum_{j=1}^s Z_{A_i,B_j}} = \frac{1}{rs}Z_{A,B} \le \frac{1}{rs} \sum_c \binom{|E_c|}{2} \le \frac{\min(|A|,|B|)}{rs}\sum_c |E_c| \le m^{-1/5} te_G(A,B) ,$$
		where the sum is over all colours $c$ and $E_c$ denotes the set of edges in $G[A,B]$ whose colour set contains~$c$. Here, we have used the assumption that the colouring is proper, which means that pairs of edges with overlapping colour sets do not share a vertex. In particular, for any colour~$c$, we have $|E_c|\le \min(|A|,|B|) \le \sqrt{|A||B|}\le m^{-1/5}rs$.
		
		We can conclude that the desired partitions exist. Finally, invoking~\eqref{individual chernoff}, we have
		\begin{align*}
			|e_{G'}(A,B) - p^t e_G(A,B)| &\le \sum_{i=1}^r \sum_{j=1}^s  \left|e_{G'}(A_i,B_j)-p^te_G(A_i,B_j) \right| \\
			&\le \sum_{i=1}^r \sum_{j=1}^s Z_{A_i,B_j} +2\sqrt{\log n} \sum_{i=1}^r \sum_{j=1}^s \sqrt{(|A_i|+|B_j|)|A_i||B_j|}\\
			&\le 9tm^{-1/5}|A||B|\sqrt{\log n},
		\end{align*}
		where the last inequality holds by~\eqref{partition} and since
		\begin{align*}
			\sum_{i=1}^r \sum_{j=1}^s \sqrt{(|A_i|+|B_j|)|A_i||B_j|} &\le \sum_{i=1}^r \sqrt{|A_i|}\sum_{j=1}^s |B_j| +  \sum_{j=1}^s \sqrt{|B_j|}\sum_{i=1}^r |A_i| \\
			& \le |B|\cdot r\sqrt{|A|/r} + |A|\cdot s\sqrt{|B|/s}\\
			&\le 4m^{-1/5}|A||B|
		\end{align*}
	(note that $\sqrt{r/|A|},\sqrt{s/|B|} \le 2 m^{-1/5}$ by definition of $r,s$).
	\end{proof}

	With the tools above at hand, we can now prove the rainbow matching lemma.
	
	\lateproof{Lemma~\ref{lem:rainbow matching}}
	We will assume that $\eps>0$ is sufficiently small and $n$ is sufficiently large.
	Note that condition \ref{cond dense} implies that $e(G)\ge \alpha n^2$ and hence $d\ge \alpha n/2$.
	
	Set $p_0:=1-\alpha/t$ and $p_i=\alpha/5t$ for $i\in [5]$.
	We first split the colours randomly as follows: for each colour independently, choose a label $i\in \{0,1,2,3,4,5\}$ at random by choosing $i$ with probability~$p_i$. 
	Let $G_i$ be the random subgraph of $G$ consisting of all edges whose colour set has been completely labelled with~$i$.
	
	We first record some important properties of these subgraphs. 
	With high probability, for every vertex $u$ and each label $i$, we have
	\begin{align}
		d_{G_i}(u)=(1\pm 2\eps) p_i^t d. \label{split degrees}
	\end{align}
	This is because $d_{G_i}(u)$ is the sum of $d_G(u)\le n$ independent Bernoulli variables and its expectation is $p_i^t d_G(u)$, so the Chernoff bound (Lemma~\ref{lem:chernoff}) implies that 
	$$\prob{|d_{G_i}(u) - p_i^t d_G(u)| \ge n^{2/3} } \le 2\exp(-2n^{1/3}).$$
	
	Moreover, by Lemma~\ref{lem:colour split}, with high probability, 
	for all $i$ and all sets $A'\In A$ and $B'\In B$ of size at least $\eps n$, we have
	\begin{align}
		e_{G_i}(A',B') = p_i^t e_G(A',B') +O(n^{11/6}). \label{split edge sets}
	\end{align}

	By the probabilistic method, there exist graphs $G_0,\dots,G_5$ which have all of the above properties.
	We will find in $G_0$ an almost perfect rainbow matching, and then use $G_1,\dots,G_5$ to turn this into a perfect rainbow matching.
	For the second step, we will need the following ``expansion'' property of the graphs~$G_i$. (This is the property which allows us to have the asymptotically optimal size condition in~\ref{cond dense}.)
	
	\begin{enumerate}[label=($\dagger$)]
		\item For any set $S\In A$ (and similarly any $S\In B$) of size at least $\eps n$, there are at least $(1-\alpha/2)d$ vertices which have at least $\eps^{1/3} |S|$ neighbours in~$S$.\label{dagger}
	\end{enumerate}
	
	To see this, let $T^-$ be the set of vertices which have fewer than $\alpha p_i^t d |S|/10n$ neighbours in~$S$, and let $T^+$ be the set of vertices which have more than $(1+\eps) p_i^t |S|$ neighbours in~$S$.
	By~\eqref{split degrees}, there are at least $(1-2\eps)p_i^t d |S|$ edges going out of~$S$. By definition of $T^-$, at most $n\cdot \alpha p_i^t d |S|/10n$ of these edges go to a vertex in~$T^-$.
	Moreover, we have $|T^+|\le \eps n$, as otherwise we can use~\eqref{split edge sets} in bounding 
	\begin{align*}
		|T^+| \cdot (1+\eps)p_i^t |S| \le e_{G_i}(T^+,S) \le p_i^t |T^+||S| + O(n^{11/6})
	\end{align*}
	to reach a contradiction.
	Hence, at least $(1-\alpha/5)p_i^t d|S|$ of the edges leaving $S$ go to vertices which are neither in $T^-$ nor in $T^+$. Since each vertex not in $T^+$ has at most $(1+\eps)p_i^t |S|$ neighbours in~$S$, there must be at least $$\frac{(1-\alpha/5)p_i^t d|S|}{(1+\eps)p_i^t |S|}\ge (1-\alpha/2)d$$ vertices where these edges go to. Each of these vertices is not in $T^-$ and hence has at least $\alpha p_i^t d |S|/10n\ge \eps^{1/3}|S|$ neighbours in~$S$. \hfill $\Box$
	
	Now, we begin with the construction of the matching.
	In the first step, we find an almost-perfect rainbow matching in~$G_0$.
	By~\eqref{split degrees}, we know that every vertex has degree $(1\pm 2\eps)p_0^td$ in~$G_0$. By condition~\ref{cond colour degrees}, all colours have degree at most $(1-\alpha)d$.
	Since $p_0^t d=(1-\alpha/t)^t d \ge (1-\alpha)d$, we can apply Corollary~\ref{cor:approx} to find a rainbow matching $M_0$ of size at least $(1-8\eps)n$.

	Now, we want to turn $M_0$ into a perfect rainbow matching. We will achieve this iteratively by increasing the size of the matching in each step by one.
	More precisely, we show that for all $\ell\in \{0,\dots,n-|M_0|\}$, there exists a rainbow matching $M_\ell$ of size $|M_0|+\ell$ in $G$ which contains at most $5\ell$ edges outside~$G_0$.
	Clearly, this is true for $\ell=0$, and ultimately $M_{n-|M_0|}$ will be a perfect rainbow matching in $G$ as desired.
	Thus, it suffices to show that given $M_\ell$ as above for some $\ell<n-|M_0|$, we can find $M_{\ell+1}$.
	
	We accomplish this by finding a short augmenting path. 
	Call a sequence $v_1,v_2,\dots,v_{2k}$ of vertices \defn{augmenting} if the following holds:
	\begin{itemize}
		\item $v_1$ and $v_{2k}$ are distinct and not covered by $M_\ell$;
		\item $\set{v_{2j}v_{2j+1}}{j\in [k-1]} \In M_\ell$;
		\item the edges in $\set{v_{2j-1}v_{2j}}{j\in [k]}\In E(G)$ are pairwise colour-disjoint and also colour-disjoint from all edges in~$M_\ell$.
	\end{itemize}
	Observe that if we have such a sequence, we can increase the size of our rainbow matching by one. For this, it is first helpful to notice that, by taking a shortest such sequence, we can assume that the vertices $v_1,v_2,\dots,v_{2k}$ are distinct. Indeed, if $v_j=v_{j'}$, then as $G$ is bipartite, $j$ and $j'$ have the same parity and hence the sequence could be shortened while preserving the augmenting property.
	When $v_1,v_2,\dots,v_{2k}$ are distinct, then it is clear that removing $\set{v_{2j}v_{2j+1}}{j\in [k-1]}$ from $M_\ell$ and adding $\set{v_{2j-1}v_{2j}}{j\in [k]}$ yields a rainbow matching of size $|M_\ell|+1$.
	
	We will now find an augmenting sequence of length 10, thus ensuring that at most 5 edges from outside $G_0$ are added to the matching. 
	The $5$ augmenting edges will come from different $G_i$'s and will thus be pairwise colour-disjoint. We also have to guarantee that they do not have a colour which is already used. Therefore,
	call a colour \defn{blocked} if it is contained in the colour set of some edge in~$M_\ell$.
	Since the inductive hypothesis gives us that $M_\ell$ contains at most $5\ell$ edges outside $G_0$, at most $5t\ell $ colours with a non-zero label are blocked. Call an edge \defn{blocked} if its colour set contains a blocked colour. Since the colouring is proper, for every vertex $u$ and each $i\in[5]$, at most $5t\ell \le 5t \cdot 8\eps n\le \sqrt{\eps} n$ edges of $G_i$ at $u$ are blocked.
	This means that the graphs $G_i$, after removing blocked edges, still have the properties \eqref{split degrees}, \eqref{split edge sets} and \ref{dagger}, up to slightly changing the parameters.

	Let $a\in A$ and $b\in B$ be two arbitrary vertices which are not covered by~$M_{\ell}$. 
	Define $B_1$ as the set of all $b_1\in N_{G_1}(a)$ such that $ab_1$ is not blocked and $b_1$ is covered by~$M_\ell$. Note that we have $$|B_1|\ge d_{G_1}(a)-\sqrt{\eps} n - 8\eps n  \ge p_1^t d/2.$$ 
	Further, let $A_2$ be the set of vertices which are matched to the vertices in $B_1$ by~$M_\ell$.
	
	Next, let $B_3$ be the set of all vertices $b_3$ for which there exists $a_2\in A_2$ such that $a_2b_3\in E(G_3)$ is not blocked and $b_3$ is covered by $M_\ell$. 
	We claim that $|B_3|\ge (1-\alpha)d$. Indeed, since $|A_2|=|B_1|\ge \eps n$, by \ref{dagger}, there exists a set $B_3'$ of size at least $(1-\alpha/2)d$ such that every $b_3\in B_3$ has at least $\eps^{1/3} |A_2|$ neighbours in~$A_2$. For each such vertex $b_3$, at most $\sqrt{\eps}n$ edges of $G_3$ are blocked, and hence there exists some $a_2\in A_2$ such that $a_2b_3\in E(G_3)$ is not blocked. By removing at most $8\eps n$ vertices from $B_3'$ which are not covered by $M_\ell$, we find the set~$B_3$ with the required properties.
	
	Finally, let $A_4$ be the set of vertices which are matched to the vertices in $B_3$ by~$M_\ell$.
	With exactly the same argument, starting from $b$, we find sets $A_1,B_2,A_3,B_4$ with the analogous properties, where $G_1$ is replaced by $G_2$ and $G_3$ is replaced by~$G_4$.
	
	\begin{figure}[ht]
		\begin{center}
			\begin{tikzpicture}
				
				\coordinate (a) at (-1,1);
				\coordinate (b) at (-1,-1);
				\coordinate (a1) at (1.4,1);
				\coordinate (b1) at (1.4,-1);
				\coordinate (a2) at (3.4,1);
				\coordinate (b2) at (3.4,-1);
				\coordinate (a3) at (6.7,1);
				\coordinate (b3) at (7,-1);
				\coordinate (a4) at (11,1);
				\coordinate (b4) at (10.7,-1);
				
				\draw (1,0.8) rectangle (2,1.2);
				\draw (3,0.8) rectangle (4,1.2);
				\draw (6,0.8) rectangle (8,1.2);
				\draw (10,0.8) rectangle (12,1.2);
				
				\draw (1,-0.8) rectangle (2,-1.2);
				\draw (3,-0.8) rectangle (4,-1.2);
				\draw (6,-0.8) rectangle (8,-1.2);
				\draw (10,-0.8) rectangle (12,-1.2);
				
				\begin{scope}[opacity=0.5]
					\draw[fill,red!30] (a) -- (1,-0.8) -- (2,-0.8) -- (a);
					\draw[fill,Cyan!30] (b) -- (1,0.8) -- (2,0.8) -- (b);
					
					\draw[fill,orange!30] (3,0.8) -- (6,-0.8) -- (8,-0.8) -- (4,0.8)--cycle;
					\draw[fill,Purple!30] (3,-0.8) -- (6,0.8) -- (8,0.8) -- (4,-0.8)--cycle;
					
					\draw[fill,ForestGreen!30] (10,0.8) -- (12,0.8) -- (12,-0.8) -- (10,-0.8)--cycle;
				\end{scope}
				
				\draw[line width=1.5pt] (a1)--(b2);
				\draw[line width=1.5pt] (b1)--(a2);
				\draw[line width=1.5pt] (a3)--(b4);
				\draw[line width=1.5pt] (b3)--(a4);

				\draw[line width=1.5pt,red] (a)--(b1);
				\draw[line width=1.5pt,orange] (a2)--(b3);
				\draw[line width=1.5pt,ForestGreen] (a4)--(b4);
				\draw[line width=1.5pt,Purple] (a3)--(b2);
				\draw[line width=1.5pt,Cyan] (a1)--(b);
				
				\draw[fill] (a) circle (2pt);
				\draw[fill] (b) circle (2pt);
				\draw[fill] (a1) circle (2pt);
				\draw[fill] (b1) circle (2pt);
				\draw[fill] (a2) circle (2pt);
				\draw[fill] (b2) circle (2pt);
				\draw[fill] (a3) circle (2pt);
				\draw[fill] (b3) circle (2pt);
				\draw[fill] (a4) circle (2pt);
				\draw[fill] (b4) circle (2pt);
				
				\node at ($(a)+(0,0.5)$) {$a$};
				\node at ($(b)+(0,-0.5)$) {$b$};
				\node at ($(a1)+(0,0.5)$) {$a_1$};
				\node at ($(b1)+(0,-0.5)$) {$b_1$};
				\node at ($(a2)+(0,0.5)$) {$a_2$};
				\node at ($(b2)+(0,-0.5)$) {$b_2$};
				\node at ($(a3)+(0,0.5)$) {$a_3$};
				\node at ($(b3)+(0,-0.5)$) {$b_3$};
				\node at ($(a4)+(0,0.5)$) {$a_4$};
				\node at ($(b4)+(0,-0.5)$) {$b_4$};
				
				\node at ($(a)+(-0.8,0)$) {$A$};
				\node at ($(b)+(-0.8,0)$) {$B$};
			\end{tikzpicture}
		\end{center}
		\caption{Sketch of the augmenting sequence. The different colours represent the graphs~$G_i$, the black edges are matching edges.}
		\label{fig:switching}
	\end{figure}
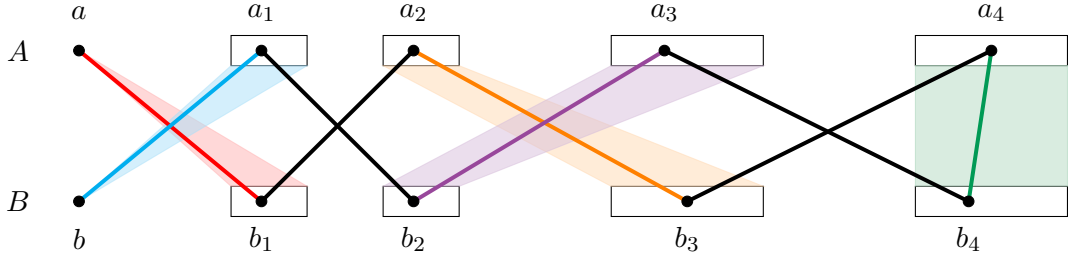

	Having defined all these sets, we can now easily find the desired augmenting sequence as follows.
	Since $A_4$ and $B_4$ have size at least $(1-\alpha)d$, we have $e_G(A_4,B_4)\ge \alpha n^2$ by~\ref{cond dense}.
	Hence, we also have $e_{G_5}(A_4,B_4)\ge p_5^t \alpha n^2-O(n^{11/6})\ge \eps^{1/3}n^2$ by~\eqref{split edge sets}. Since in total there are at most $\sqrt{\eps}n^2$ blocked edges, we can find an edge $a_4b_4\in E(G_5)$ between $A_4$ and $B_4$ which is not blocked.
	Let $b_3\in B_3$ be the match of $a_4$. By definition of $B_3$, there exists $a_2\in A_2$ such that $a_2b_3\in E(G_3)$ is not blocked. Finally, let $b_1$ be the match of $a_2$ and note that by definition of $B_1$, we have that $ab_1\in E(G_1)$ is not blocked.
	Similarly, starting from $b_4$, find vertices $a_3$, $b_2$, and $a_1$. Then $a,b_1,a_2,b_3,a_4,b_4,a_3,b_2,a_1,b$ is an augmenting sequence as desired (cf.~Figure~\ref{fig:switching}).
	\endproof

	\section{Upper bound}\label{sec:construction}
	
	In this section, we establish our upper bound on $\qc(n)$, by constructing partial $n$-queens configurations which cannot be completed.
	Before proving the upper bound given in Theorem~\ref{thm:main}, we give the following very simple argument which shows that $\qc(n)\le n/3+O(1)$.
	Assume for simplicity that $n$ is divisible by $3$. Partition the chessboard into 9 squares of side length $n/3$. Now, simply view the central box as an $n/3\times n/3$ chessboard, and put an $n/3$-queens configuration there. (This is possible whenever $n/3\ge 4$.) Clearly, this is a partial $n$-queens configuration of the original chessboard.
	We claim that it cannot be completed. Indeed, no additional queen can be placed in any of the 4 boxes which share a side with the central box. Hence, if it was possible to place $2n/3$ additional queens, then they all have to be placed into the 4 corner boxes. Hence, we can assume that at least $n/3$ queens are placed in two corner boxes which lie diagonally opposite to each other, say top left and right bottom. However, these two corner boxes together with the central box are covered by $2n/3-1$ parallel diagonals. Since $n/3$ of these diagonals are already blocked by queens in the central box, it is impossible to place another $n/3$ queens into the two corner boxes.

	Now, we continue with the proof of our improved bound.
	While the construction is similar (we place an $m$-queens configuration into the middle of the chessboard), a more elaborate argument is needed to show that it cannot be completed.
	In fact, we formulate a very general setup based on linear programming duality which can also be used to improve our bound.

	Given the $n\times n$ chessboard, a \defn{line weighting} is a function $w\colon \cL \to [0,1]$. Its \defn{value} is simply $\sum_{L\in \cL} w(L)$. 
	We say that $w$ \defn{covers} a square $(i,j)$ if the sum of the weights of the lines containing that square is at least~$1$, that is, $$\sum_{L\in \cL\colon (i,j)\in L} w(L)= w(R_i)+w(C_j)+w(D^+_{i+j-(n+1)})+w(D^-_{i-j})\ge 1.$$ For a subset $\Lambda\In [n]\times [n]$, we say that $w$ \defn{covers} $\Lambda$ if it covers every element of~$\Lambda$.
	
	\begin{prop}\label{prop:dual}
		Let $Q'$ be a partial $n$-queens configuration and let $\Lambda\In [n]\times [n]$ be the set of squares not attacked by~$Q'$. If there exists a line weighting which covers $\Lambda$ and has value less than $n-|Q'|$, then $Q'$ cannot be completed.
	\end{prop}
	
	\begin{proof}
		Let $w$ be a line weighting which covers $\Lambda$. Suppose that $Q'$ can be completed to an $n$-queens configuration $Q$. Hence, $Q\setminus Q'\In \Lambda$ has size $n-|Q'|$. For each $(i,j)\in Q\setminus Q'$, we have $\sum_{L\in \cL\colon (i,j)\in L} w(L)\ge 1$. By summing this inequality over all $(i,j)\in Q\setminus Q'$, since any line will only appear at most once on the left hand side, we see that the value of $w$ is at least $n-|Q'|$, which yields a contradiction.
	\end{proof}

	We will use Proposition~\ref{prop:dual} to prove that our constructed partial $n$-queens configuration cannot be completed. For this, we need to find a suitable line weighting which certifies this.
	As mentioned before, in our construction, $Q'$ will be an $m$-queens configuration in the centre of the chessboard. This restricts $\Lambda$ to the four corner boxes of side length $(n-m)/2$.
	We will first define a line weighting $\hat{w}$ which covers all elements of these four corner boxes.
	Subsequently, we observe that we can actually set the weight of all diagonals which contain an element of $Q'$ to $0$, thus significantly reducing the value of the line weighting. 
	In the weighting $\hat{w}$, the diagonals which are close to the main diagonals will have larger weight than those which are further away.
	Hence, the value of our line weighting is reduced the most if the elements of $Q'$ are ``close'' to the main diagonals.
	It turns out that a well-known construction of (toroidal) $n$-queens configurations (cf.~\cite{BS:09}) has the property that many queens are close to the main diagonals.

	\begin{prop}\label{prop:central box}
		For $n\equiv 1 \mod{6}$, there exists an $n$-queens configuration $Q$ such that 
		$$\sum_{(i,j)\in Q} |i+j-(n+1)|+|i-j| \le 2n^2/3+O(n).$$
	\end{prop}
	
	\begin{proof}
		For each $i\in \{1,\dots,(n-1)/2\}$, we place a queen on $(i,2i)$. For each $i\in \{(n+1)/2,\dots,n\}$, we place a queen on $(i,2i-n)$.
		The quantity of interest then becomes
		\begin{align*}
			&\sum_{i=1}^{\frac{n-1}{2}} |3i-n-1| + |-i| + \sum_{i=\frac{n+1}{2}}^{n} |3i-2n-1| + |-i+n|\\
			&=n^2 \left[ \int_{0}^{1/2}|3x-1|+x \;\mathrm{d}x + \int_{1/2}^{1}|3x-2|+|1-x| \;\mathrm{d}x  \right] +O(n)\\
			&=2n^2/3 + O(n).
		\end{align*}
	\end{proof}

	\lateproof{Theorem~\ref{thm:main}, upper bound}
	Suppose $n$ is large enough. Let us first assume that $n$ is odd. Let $m$ be the largest integer which is at most $0.241n$ and $\equiv 1\mod{6}$. (We will actually prove that $\qc(n)<m \le 0.241n<n/4$.) 
	Set $t=(n-m)/2$.
	
	Let $Q$ be the $m$-queens configuration from Proposition~\ref{prop:central box}. We embed it in the centre of the chessboard, more precisely, we let $Q'=\set{(i+t,j+t)}{(i,j)\in Q}$.
	Obviously, $Q'$ is a partial $n$-queens configuration. We want to show that it cannot be completed.
	
	Let $\Lambda\In [n]\times [n]$ be the set of squares not attacked by~$Q'$.
	Clearly, $\Lambda$ is contained in the set $$\hat{\Lambda}= \set{(i,j)}{i,j\in \{1,\dots,t\}\cup \{n-t+1,\dots,n\}}$$ which consists of the four $t\times t$ corners of the chessboard.
	We will first define a line weighting which covers $\hat{\Lambda}$. For this, define
	$\hat{w}$ as follows:
	for each $i\in \{1,\dots,t\}$, set 
	\begin{align}
		\hat{w}(R_i) = \hat{w}(C_i)= \hat{w}(R_{n+1-i}) = \hat{w}(C_{n+1-i}) := \left|\frac{i}{t+1}-\frac12\right|.
	\end{align}
	Moreover, for each $k\in \{-(t-1),\dots,t-1\}$, set
	\begin{align}
		\hat{w}(D^+_k) = \hat{w}(D^-_k) = 1-\left|\frac{k}{t+1}\right|. 
	\end{align}
	Set the weight of all other lines to~$0$.
	
	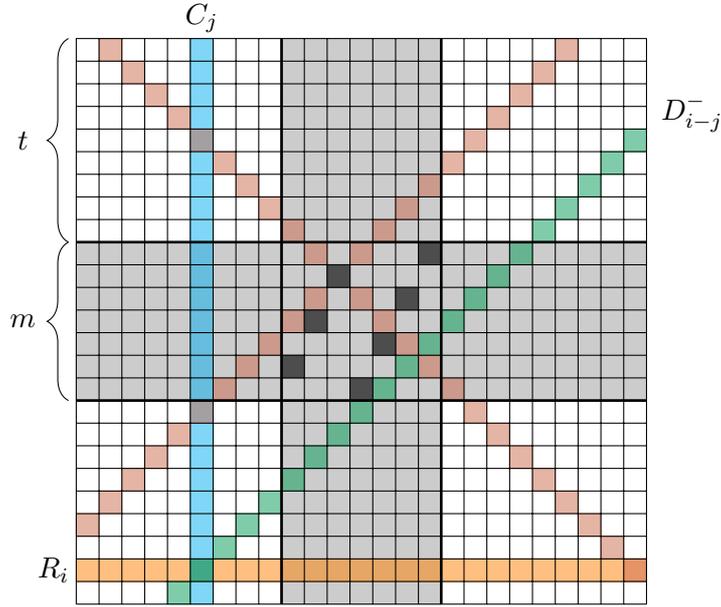
\begin{figure}[ht]
		\begin{center}
			\begin{tikzpicture}[scale=0.3]
				
				\draw[fill,gray!40] (0,9) rectangle (25,16);
				\draw[fill,gray!40] (9,0) rectangle (16,25);
				
				\begin{scope}[opacity=0.5]
					\draw[fill,orange] (0,1) rectangle (25,2);
					\draw[fill,Cyan] (5,0) rectangle (6,25);
					\foreach \x in {0,...,20}
					\draw[fill,ForestGreen]  (\x+4,\x) rectangle (\x+5,\x+1);
					
					\foreach \x in {0,...,21}
					\draw[fill,BrickRed!70]  (\x,\x+3) rectangle (\x+1,\x+4);
					\foreach \x in {1,...,24}
					\draw[fill,BrickRed!70]  (\x,26-\x) rectangle (\x+1,25-\x);
					
				\end{scope}
				
				\draw[fill,black!70] (9,10) rectangle (10,11);
				\draw[fill,black!70] (10,12) rectangle (11,13);
				\draw[fill,black!70] (11,14) rectangle (12,15);
				\draw[fill,black!70] (12,9) rectangle (13,10);
				\draw[fill,black!70] (13,11) rectangle (14,12);
				\draw[fill,black!70] (14,13) rectangle (15,14);
				\draw[fill,black!70] (15,15) rectangle (16,16);

				
				\draw[step=1,black,thin] (0,0) grid (25,25);
				\draw[line width=1pt] (0,9) -- (25,9);
				\draw[line width=1pt] (9,0) -- (9,25);
				\draw[line width=1pt] (16,0) -- (16,25);
				\draw[line width=1pt] (0,16) -- (25,16);

				\node at (-1,1.5) {$R_i$};
				\node at (5.5,25.9) {$C_j$};
				\node at (27,21.7) {$D^-_{i-j}$};
				\draw[decorate,decoration={brace,amplitude=8pt},xshift=-4pt,yshift=0pt]
				(-0.2,9) -- (-0.2,16) node [black,midway,xshift=-0.6cm] {$m$};
				\draw[decorate,decoration={brace,amplitude=8pt},xshift=-4pt,yshift=0pt]
				(-0.2,16) -- (-0.2,25) node [black,midway,xshift=-0.6cm] {$t$};
			\end{tikzpicture}
		\end{center}
		\caption{The black squares in the centre represent the partial configuration~$Q'$. The grey squares share a row or column with $Q'$, leaving the four $t\times t$ corners. Each square there is covered by the weight of its row, column and the longer diagonal. Ultimately, we can ``remove'' the diagonals which pass through~$Q'$.}
	\end{figure}
	
	We claim that $\hat{w}$ covers $\hat{\Lambda}$. By symmetry, it suffices to consider $(i,j)\in [t]\times[t]$. We then have
	\begin{align*}
		\sum_{L\in \cL\colon (i,j)\in L} \hat{w}(L) &\ge \hat{w}(R_i)+\hat{w}(C_j) + \hat{w}(D^-_{i-j}) \\
		&= \left|\frac{i}{t+1}-\frac12\right|+\left|\frac{j}{t+1}-\frac12\right|+1-\left|\frac{i-j}{t+1}\right|\ge 1,
	\end{align*}
	where the last inequality holds by the triangle inequality. This proves the claim.
	
	Notice also that $\hat{w}$ has value $3t+O(1)$, where the $2t$ rows and $2t$ columns contribute $1/4+O(1/t)$ on average, and the $4t$ diagonals contribute $1/2+O(1/t)$ on average.
	
	Now, observe that we can set the weight of all diagonals which contain some queen from $Q'$ to $0$, because no element of such a diagonal is in~$\Lambda$.
	The obtained line weighting $w$ then still covers $\Lambda$, and its value is reduced (with respect to the value of $\hat{w}$) by
	\begin{align*}
		\sum_{(i,j)\in Q'} \hat{w}(D^+_{i+j-(n+1)}) + \hat{w}(D^-_{i-j})&= 2|Q'|-\sum_{(i,j)\in Q'} \left|\frac{i+j-(n+1)}{t+1}\right| + \left|\frac{i-j}{t+1}\right|\\
		&=2m-\frac{1}{t+1} \sum_{(i,j)\in Q'} \left|i+j-(n+1)\right| + \left|i-j\right|.
	\end{align*}
	By Proposition~\ref{prop:central box}, we have 
	\begin{align*}
		\sum_{(i,j)\in Q'} \left|i+j-(n+1)\right| + \left|i-j\right|&= \sum_{(i,j)\in Q} \left|(i+t)+(j+t)-(n+1)\right| + \left|(i+t)-(j+t)\right|\\
		&= \sum_{(i,j)\in Q} \left|i+j-(m+1)\right| + \left|i-j\right|\\
		&\le 2m^2/3+O(m).
	\end{align*}
	Altogether, we conclude that the value of $w$ is at most $$3t-2m + \frac{2m^2}{3t}+O(1).$$
	If this value is smaller than $n-|Q'|=2t$, then Proposition~\ref{prop:dual} certifies that $Q'$ cannot be completed.
	This condition informed our choice of~$m$. It is straightforward to check that with $m=0.241n+O(1)$ and $t=(n-m)/2$ we indeed have the desired strict inequality.
	
	Finally, if $n$ is even, we can simply repeat the above construction with $n'=n-1$, and add weight $1$ to $R_n$ and $C_n$ to get the desired weighting. 
	\endproof

	\section{Concluding remarks}\label{sec:remarks}
	
	In this paper, we studied the $n$-queens completion problem, showing that if at most $n/60$ queens are placed on an $n\times n$ chessboard, mutually non-attacking, then they can be completed to an $n$-queens configuration.
	We also provided a construction of a partial configuration of size roughly $n/4$ which cannot be completed.
	It would be very interesting to close the gap between the lower and upper bound.
	
	\begin{problem}\label{prob:main}
		Determine the $n$-queens completion threshold $\qc(n)$ asymptotically.
	\end{problem}
	
	In fact, note that our proof of the upper bound (via Proposition~\ref{prop:dual}) shows that $Q'$ does not even have a fractional completion. In more detail, given a partial configuration $Q'$, let $\Lambda$ as before be the set of squares which are not attacked by~$Q'$. We say that $Q'$ has a \defn{fractional completion} if there exists a non-negative weight function on $\Lambda$ with the total weight $n-|Q'|$ such that the weight along each line is at most~$1$. Similar to $\qc(n)$, one can define $\qc^*(n)$ as the \defn{fractional $n$-queens completion threshold}, that is, the largest integer with the property that any partial $n$-queens configuration on an $n\times n$ chessboard with at most $\qc^*(n)$ queens has a fractional completion. Obviously, $\qc(n)\le \qc^*(n)$, and (by virtue of our proof) the upper bound in Theorem~\ref{thm:main} also holds for~$\qc^*(n)$.
	We conjecture that in fact $\qc(n)$ and $\qc^*(n)$ are asymptotically equal. This breaks Problem~\ref{prob:main} into two natural sub-problems, both of which are interesting on their own.
	Show that $\qc(n)\sim \qc^*(n)$ and determine $\qc^*(n)$.
	
	So far, we were mainly concerned with finding one completion of a given partial $n$-queens configuration~$Q'$ (or showing that none exists).
	However, it is also natural to ask how many distinct completions there are. For instance, Nauck's original configuration (Figure~\ref{fig:nauck}) has 2 completions.
	A trivial upper bound is the number of all $n$-queens configurations, which is at most $n!\le n^n$.
	One can adapt our proof to show that not only can any partial configuration of size at most $n/60$ be completed, but even that there are super-exponentially many completions.
	\begin{theorem}\label{thm:counting}
		For large enough $n$, any partial $n$-queens configuration of size at most $n/60$ can be completed in at least $n^{\Omega(n)}$ different ways.
	\end{theorem}

This follows immediately (by following the proof of Theorem~\ref{thm:main}) from a counting version of our rainbow matching lemma. Namely, in Lemma~\ref{lem:rainbow matching}, under the same assumptions, there exist $n^{\Omega(n)}$ perfect rainbow matchings. The approach to prove this is quite standard by now and has been used for many counting problems in combinatorics (and also in~\cite{BK:21,LS:21,simkin:21} for counting $n$-queens configurations), so we omit the details. Roughly speaking, one requires a counting version of Theorem~\ref{thm:nibble} which in the proof of Lemma~\ref{lem:rainbow matching} yields $n^{\Omega(n)}$ different approximate rainbow matchings~$M_0$. Then, by inspecting our proof, each of these approximate rainbow matchings can be augmented to a perfect rainbow matching, and at most $n^{O(\eps n)}$ approximate matchings produce the same perfect matching, which gives the desired result.

	We believe that the statement Theorem~\ref{thm:counting} remains true if $n/60$ is replaced by $(1-o(1))\qc(n)$. Also, note that the counting result is quite rough. It would be interesting to determine the number of completions for a given partial configuration $Q'$ more accurately, say within subexponential factors. Perhaps some of the techniques introduced in~\cite{simkin:21} are adaptable to this ``generalized $n$-queens problem''.

	Finally, it might also be interesting to look at \defn{embeddings} instead of \defn{completions}. Generally speaking, if some partial instance of a combinatorial structure is not completable, one might still wonder if it can be embedded into a complete instance of larger order, and if so, by how much the order has to be increased for this to be possible.
	For instance, while there are partial Latin squares of order $n$ with $n$ cells filled in which cannot be completed (to a Latin square of order~$n$), Evans~\cite{evans:60} showed that any partial Latin square of order $n$ can be embedded into a Latin square of order at most $2n$. The corresponding result for Steiner triple systems was obtained by Bryant and Horsley~\cite{BH:09}, who thereby proved a conjecture of Lindner. For a general discussion of such problems in graph theory, see~\cite{NSW:20}.

	Given a partial $n$-queens configuration $Q'$, we thus ask what is the smallest $n^*$ such that $Q'$ can be embedded into an $n^*$-queens configuration~$Q$. Here, by embedding we mean that the $n\times n$ chessboard on which $Q'$ is given is a subsquare of the larger board, that is, there exist $i_0,j_0$ such that $(i+i_0,j+j_0)\in Q$ for all $(i,j)\in Q'$. Note that $n^*=n$ if and only if $Q'$ is completable. The existence of $n^*$ was first shown in~\cite{GJN:17} (with an implicit upper bound linear in $n$), which also follows immediately from our Theorem~\ref{thm:main}. 

	\section*{Acknowledgment}
We thank two anonymous reviewers for a careful reading of our paper and many helpful comments.
	
	
	\providecommand{\bysame}{\leavevmode\hbox to3em{\hrulefill}\thinspace}
	\providecommand{\MR}{\relax\ifhmode\unskip\space\fi MR }
	\providecommand{\MRhref}[2]{%
		\href{http://www.ams.org/mathscinet-getitem?mr=#1}{#2}
	}
	\providecommand{\href}[2]{#2}


\end{document}